\documentclass[a4paper,10pt, oneside]{article}
\usepackage{amsmath, amssymb, amsthm,graphicx}
\usepackage[font=small,labelfont=md,textfont=it]{caption}
\usepackage{floatrow,float}
\usepackage[titletoc, title]{appendix}
\usepackage[hyperpageref]{backref}
\usepackage[colorlinks,linkcolor=blue,citecolor=blue]{hyperref}
\usepackage{longtable}
\usepackage{pgfplots}
\usepackage{diagbox}
\usepackage{booktabs,makecell,multirow}
\usepackage[capitalise, nosort]{cleveref}
\usepackage{cases,color}
\usepackage{algorithm}
\usepackage{algorithmic}
\usepackage{amsmath,bm}
\usepackage{geometry}

\usepackage{tcolorbox}

\crefname{equation}{}{}
\crefname{lem}{Lemma}{Lemmas}
\crefname{thm}{Theorem}{Theorems}

\newcommand{\dd}{\,{\rm d}}

\newcommand{\bs}{\boldsymbol}

\newcommand{\dual}[1]{\left\langle {#1} \right\rangle}
\newcommand{\proxi}[0]{ {\bf prox}}

\newcommand{\argmin}[0]{ {\mathrm {argmin}\,}}

\newcommand{\nm}[1]{\left\lVert {#1} \right\rVert}

\newcommand{\ssnm}[1]
{
	\left\vert\kern-0.25ex
	\left\vert\kern-0.25ex
	\left\vert
	{#1}
	\right\vert\kern-0.25ex
	\right\vert\kern-0.25ex
	\right\vert
}

\makeatletter
\def\spher@harm#1{%
	\vbox{\hbox{%
			\offinterlineskip
			\valign{&\hb@xt@2\p@{\hss$##$\hss}\vskip.2ex\cr#1\crcr}%
		}\vskip-.36ex}%
}
\def\gshone{\spher@harm{.}}
\def\gshtwo{\spher@harm{.&.}}
\def\gshthree{\spher@harm{.&.&.}}
\let\gsh\spher@harm
\makeatother

\newtheorem{lem}{Lemma}[section]
\newtheorem{rem}{Remark}[section]
\newtheorem{thm}{Theorem}[section]

\newcounter{mnote}
\let\oldmarginpar\marginpar
\renewcommand\marginpar[1]
{\-\oldmarginpar[\raggedleft\footnotesize #1]{\raggedright\footnotesize #1}}

\makeatletter\def\@captype{table}\makeatother

\newfloatcommand{capbtabbox}{table}[][\FBwidth]
\begin{document}
%
%
	
		\title{First order optimization methods based on Hessian-driven Nesterov accelerated gradient flow
		\thanks{Hao Luo was supported by the China Scholarship Council (CSC) joint Ph.D. student scholarship (Grant 201806240132).}
	}

	\author{	
	Long Chen \thanks{Email: chenlong@math.uci.edu},
	Hao Luo \thanks{Corresponding author. Email: galeolev@foxmail.com}. \\
}

\date{}

\maketitle
%
%
%

\begin{abstract}
A novel dynamical inertial Newton system, which is called Hessian-driven Nesterov accelerated gradient (H-NAG) flow is proposed. Convergence of the continuous trajectory are established via tailored Lyapunov function, and new first-order accelerated optimization methods are proposed from ODE solvers. It is shown that (semi-)implicit schemes can always achieve linear rate and explicit schemes have the optimal(accelerated) rates for convex and strongly convex objectives. In particular, Nesterov's optimal method is recovered from an explicit scheme for our H-NAG flow. Furthermore, accelerated splitting algorithms for composite optimization problems are also developed. 
\end{abstract}
	
%
	




 
\section{Introduction}
\label{sec:intro}
In this paper, we introduce the Hessian-based Nesterov 
accelerated gradient (H-NAG) flow:
\begin{equation}\label{eq:ode-agf-H}
	\gamma x'' +(\gamma+\mu)x'+
	(1+\mu \beta + \gamma \beta')\nabla f(x) +\gamma\beta\nabla^2 f(x)x'=0,
\end{equation}
where $x = x(t)$ is a $V$-valued function of time variable $t$ and $(\cdot)'$ is the derivative taking respect to $t$, $f:V\to\mathbb R$ is a $\mathcal C^2$ and convex function defined on the Hilbert space $V$ and the damping coefficient $\gamma(t)$ is dynamically changing by $\gamma' = \mu - \gamma,\,\mu\geqslant 0$.
The additional damping coefficient $\beta(t)$ in front of the Hessian is nonnegative. Note that \cref{eq:ode-agf-H} belongs to the class of dynamical inertial Newton (DIN) system introduced recently in \cite{attouch_first-order_2019}.

When choosing vanishing damping $\beta = 0$, \cref{eq:ode-agf-H} reduces to Nesterov accelerated gradient (NAG) flow proposed in our recent work \cite{luo_chen_2019_from}
\begin{equation}\label{eq:ode-agf}
	\gamma x''+(\gamma+\mu)x'+\nabla f(x) = 0,
\end{equation}
or equivalently, the first-order ODE system
\begin{equation}\label{eq:NAG-sys}
	\left\{
	\begin{split}
		x' = {}&v-x,\\
		\gamma v' = {}&\mu(x-v)-\nabla f(x),\\
		\gamma ' = {}& \mu - \gamma.
	\end{split}
	\right.
\end{equation}
In \cite{luo_chen_2019_from}, the presented numerical discretizations with an extra gradient step for NAG flow \cref{eq:NAG-sys} lead to old and new accelerated schemes and can recover exactly Nesterov's optimal method \cite[Chapter 2]{Nesterov:2013Introductory} for both convex ($\mu = 0$)
and strongly convex cases ($\mu > 0$) in a unified framework. When applied to composite convex optimization, our methods can recover FISTA \cite{beck_fast_2009} for convex case and give new accelerated proximal gradient methods for strongly convex case. Compared to recent ODE models  \cite{Siegel:2019,Su;Boyd;Candes:2016differential,Wilson:2018}  for studying accelerated gradient methods which usually treat convex and strongly convex cases separately, our unified analysis in \cite{luo_chen_2019_from} is due to the introduction of the dynamic damping coefficient $\gamma' = \mu - \gamma$, which brings the effect of time rescaling.

When $\beta >0$, the added Hessian-driven damping in \cref{eq:ode-agf-H} will neutralize the possible transversal oscillation occurred in the accelerated gradient method; see \cite[Figure 1]{attouch_first-order_2019} for illustration. Particularly, if $\beta \gg 1$, then the flow behaves like the continuous Newton's flow \cite{alvarez_dynamical_1998}. A direct discretization based on Hessian is restrictive and expensive since requiring $f\in \mathcal C^2$ and the cost to compute the Hessian matrix and its inverse. 

Instead we will write \cref{eq:ode-agf-H} as a first order system
\begin{equation}\label{eq:HNAG}
	\left\{
	\begin{split}
		x' = {}&v-x-\beta\nabla f(x),\\
		\gamma v' = {}&\mu(x-v)-\nabla f(x),\\
				\gamma ' = {}& \mu - \gamma,
	\end{split}
	\right.
\end{equation}
in which Hessian disappears. This agrees with the most remarkable feature of  the dynamical inertial Newton model discovered in \cite{Alvarez;Attouch;Bolte;Redont:2002second-order}. Now \cref{eq:HNAG} is well defined for $f\in \mathcal C^1$ and 
can be further generalized to non-smooth setting by replacing gradient with sub-gradient \cite{attouch_second-order_2012}. 

\subsection{Main results}
We first consider smooth and $\mu$-convex ($\mu\geqslant 0$, cf. \eqref{eq:mu}) function $f$ with $L$-Lipschitz gradient. Let $(x(t), v(t),\gamma(t))$ be the solution of \cref{eq:HNAG} and denote by $x^*$ a global minimum point of $f$. By introducing the Lyapunov function
\begin{equation}\label{eq:Lt-intro}
\mathcal L(t) = f(x(t)) - f(x^*) + \frac{\gamma(t)}{2}\nm{v(t)-x^*}^2,\quad t\geqslant 0,
\end{equation}
we shall first establish the exponential decay property
\begin{equation}\label{eq:conv-Lt-HNAG}
\mathcal L(t)+	\int_{0}^{t}e^{s-t}\beta(s)
\nm{\nabla f(x(s))}^2
\!\dd s\leqslant e^{-t}\mathcal L(0).
\end{equation}

Then we propose several implicit and explicit schemes for \cref{eq:HNAG} to get a sequence of $\{ (x_k, v_k, \gamma_k)\}$ and establish the convergence  via the discrete analogue of \cref{eq:Lt-intro}
\[
\mathcal L_k={}
f(x_k)-f(x^*) +
\frac{\gamma_k}{2}\nm{v_k-x^*}^2,\quad k\geqslant 0.
\]

For a semi-implicit scheme (proximal method), we shall prove 
\[
\mathcal L_k+	\lambda_k\sum_{i=0}^{k-1}
\frac{\alpha_{i}^2}{\lambda_i\gamma_i}
\nm{\nabla f(x_{i+1})}^2
\leqslant\lambda_k\mathcal L_0,
\]
where the sequence $\{\lambda_k\}$ is defined by that
\begin{equation}\label{eq:lambdak-intro}
	\lambda_0=1,\quad\lambda_k = 
	\prod_{i=0}^{k-1}\frac{1}{1+\alpha_i},\quad k\geqslant 1.
\end{equation}
We easily obtain the linear convergence rate as long as 
the time step size $\alpha_k$ is bounded below,

Proximal method relies on a fast solver of a regularized problem which may not be available. To be practical, we propose an explicit scheme (cf. \cref{eq:ex-HNAG}) for solving \cref{eq:HNAG}. This scheme has been rewritten in the following algorithm style.
\begin{algorithm}[H]
	\caption{HNAG method for minimizing $f$}
	\label{algo:HNAG}
	\begin{algorithmic}[1] 
		\REQUIRE  $ \gamma_0>0$ and $x_0,v_0\in V$.
		\FOR{$k=0,1,\ldots$}
		\STATE Compute $\alpha_k, \beta_k$ by $\displaystyle \alpha_k = \sqrt{\frac{\gamma_k}{L}}, \quad \beta_k = \frac{1}{L\alpha_k}$.
		\smallskip
		\STATE Update $\displaystyle x_{k+1} = \frac{1}{1+\alpha_k} \big [ x_k + \alpha_k v_k - \alpha_k\beta_k \nabla f(x_k) \big ]$.
		\smallskip			
		\STATE Update $\displaystyle v_{k+1} = 
		\frac{1}{\gamma_k+\mu\alpha_k}
		\big [\gamma_kv_k+\mu \alpha_kx_{k+1}- \alpha_k \nabla f(x_{k+1}) \big]$.
		\smallskip
		\STATE Update $\gamma_{k+1} = (\gamma_k+\mu\alpha_k)/(1+\alpha_k)$.			
		\ENDFOR
	\end{algorithmic}
\end{algorithm}

We shall prove the convergence result for \cref{algo:HNAG}:
\begin{equation}\label{eq:conv-ex-acc}
	\mathcal L_k+\frac{1}{2L}\sum_{i=0}^{k-1}
	\frac{\lambda_k}{\lambda_i}\nm{\nabla f(x_{i})}^2
	\leqslant\lambda_k\mathcal L_0,
\end{equation}
where $\lambda_k$ is introduce by \cref{eq:lambdak-intro} and has the estimate
\begin{equation*}
	\lambda_k\leqslant \min\left\{
	8L\left(2\sqrt{2L}+\sqrt{\gamma_0}k\right)^{-2},\,
	\left(1+\sqrt{\min\{\gamma_0,\mu\}/L}\right)^{-k}
	\right\}.
\end{equation*}
Note that the above rate of convergence is optimal in the sense of the optimization complexity theory \cite{Nemirovskii1983,Nesterov:2013Introductory}. Furthermore \cref{eq:conv-ex-acc}  promises faster convergence rate for the norm of the gradient; see \cref{rem:norm-gd}.

In our recent work \cite{luo_chen_2019_from}, we verified that NAG method can be recovered from an explicit scheme for NAG flow \cref{eq:NAG-sys} with an extra gradient descent step which is not a discretization of the ODE \cref{eq:NAG-sys}. In this paper, we further show that NAG method is actually an explicit scheme for H-NAG flow \cref{eq:HNAG} {\it without} 
extra gradient step. From this point of view, our H-NAG model \cref{eq:HNAG} offers better explanation and understanding for Nesterov's accelerated gradient method than NAG flow \cref{eq:NAG-sys} does.

We finally propose a new splitting method (cf. \cref{eq:ex-HODE})
for composite convex optimization $f = h + g$. Here, the objective $f$ is $\mu$-convex with $\mu\geqslant 0$, $h$ is a smooth convex function with $L$-Lipschitz gradient and $g$ is convex but non-smooth. 

\begin{algorithm}[H]
	\caption{HNAG method for minimizing $f = h + g$}
	\label{algo:HNAG-comp}
	\begin{algorithmic}[1] 
		\REQUIRE  $\gamma_0>0$ and $x_0,v_0\in V$.
		\FOR{$k=0,1,\ldots$}
		\STATE Compute $\alpha_k, \beta_k$ by $\displaystyle \alpha_k 
		= \sqrt{\frac{\gamma_k}{L}}, \quad \beta_k = \frac{1}{L\alpha_k}$.
		\smallskip
		\STATE Set $\displaystyle z_{k} = \frac{1}{1+\alpha_k}\big[x_k + \alpha_k v_k - \alpha_k\beta_k \nabla h(x_k)\big]$.
		\smallskip
		\STATE Update $x_{k+1} =\proxi_{s g}(z_{k})$ with $s =\alpha_k\beta_k/(1+\alpha_k) $.
		\smallskip			
		\STATE Set $\displaystyle p_{k+1} = \frac{1}{\beta_k}\big [v_{k}-x_{k+1}-\beta_k\nabla h(x_{k})-(x_{k+1}-x_{k})/\alpha_k\big ]\in\partial g(x_{k+1})$.			
		\smallskip
		\STATE Update $\displaystyle v_{k+1} = 
		\frac{1}{\gamma_k+\mu\alpha_k}
		\big [\gamma_kv_k+\mu \alpha_kx_{k+1}-
		\alpha_k \nabla h(x_{k+1}) -\alpha_k  p_{k+1}\big]$.		
		\smallskip
		\STATE Update $\gamma_{k+1} = (\gamma_k+\mu\alpha_k)/(1+\alpha_k)$.						 
		\ENDFOR
	\end{algorithmic}
\end{algorithm}

Observe that \cref{algo:HNAG-comp} is almost identical to \cref{algo:HNAG} except we use proximal operator for the non-smooth convex function $g$. For any $\lambda>0$, the proximal operator $\proxi_{\lambda g}$ is defined by that \cite{attouch_variational_2014,parikh_proximal_2014}
\begin{equation*}
\proxi_{\lambda g}(x) = 
\inf_{y\in V}
\left(
g(y) + \frac{1}{2\lambda}\nm{x-y}^2
\right)\quad\forall\,x\in V.
\end{equation*}
For \cref{algo:HNAG-comp}, 
the following 
accelerated convergence rate has been established
\begin{equation*}
	\mathcal L_k
	\leqslant\mathcal L_0\times \min\left\{
	8L\left(2\sqrt{2L}+\sqrt{\gamma_0}k\right)^{-2},\,
	\left(1+\sqrt{\min\{\gamma_0,\mu\}/L}\right)^{-k}
	\right\},
\end{equation*}
and with an alternative choice for $\alpha_k$ and $\beta_k$, 
we can obtain faster convergence rate for the norm of 
(sub-)gradient; see \cref{rem:alter-split}.

\subsection{Related work and main contribution}
The most relevant works are \cite{attouch_first-order_2019,Shi:2018} where ODE models with Hessian driven damping are studied. We defer to \S \ref{sec:review} for a detailed literature review. 

We follow closely to \cite{attouch_first-order_2019,Shi:2018}. Namely we first analyze the ODE using a Lyapunov function, then construct optimization algorithms from numerical discretizations of this ODE, and use a discrete Lyapunov function to study the convergence of the proposed algorithms. 

Our main contribution is a relatively simple ODE model with dynamic damping coefficient $\gamma$ which can handle both the convex case ($\mu = 0$) and strongly convex case ($\mu > 0$) in a unified way. Our continuous and discrete Lyapunov functions are also relatively simple so that most calculation is straightforward. 

Another major contribution is a simplified Lyapunov analysis by introducing the strong Lyapunov property cf. \eqref{eq:LG}, which simplifies the heavy algebraic manipulation in~\cite{attouch_first-order_2019,Shi:2018,shi_acceleration_nodate,Su;Boyd;Candes:2016differential}. 
We believe our translation of results from continuous-time ODE to discrete algorithms is more transparent and helpful for the design and analysis of existing and new optimization methods. For example, we successfully developed splitting algorithms for composite optimization problems not restricted to a special case as considered in \cite{attouch_first-order_2019}.
\subsection{Function class}
Throughout this paper, assume $V$ is equipped with the inner product $(\cdot,\cdot)$ and the norm $\nm{\cdot}=(\cdot,\cdot)^{1/2}$.
We use 
$\dual{\cdot,\cdot}$ to denote the duality pair 
between $V^*$ and $V$, 
where $V^*$ is the dual space of $V$.  Denote by $\mathcal F_L^{1}$ the set of all convex functions $f\in\mathcal C^1$ with $L$-Lipschitz 
continuous gradient: 
\begin{equation*}
\|\nabla f(x) - \nabla f(y)\|_{*} 
\leqslant L\| x - y \| \quad \forall\, x,y \in V,
\end{equation*}
where $\nm{\cdot}_{*}$ denotes the 
dual norm on $V^*$. We say that $f$ is $\mu$-convex if there exists 
$\mu\geqslant 0$ such that
\begin{equation}\label{eq:mu}
f(x) - f(y) - \langle p, x - 
y \rangle  \geqslant \frac{\mu}{2} \| x- y \|^2
\quad \forall\,p\in  \partial f(y),
\end{equation}
for all $ x,y \in V$, where the sub-gradient $\partial f(y)$ 
of $f$ at $y\in V$ is defined by that
\begin{equation}\label{eq:sub}
\partial f(y) :=
\left\{p\in V^*:\,f(x)\geqslant f(y)+\dual{p,x-y}\quad
\,\forall\, x\in V\right\}.
\end{equation}
We use $\mathcal S_\mu^0$ to denote the set of all 
$\mu$-convex functions. In addition, we set $\mathcal S_{\mu}^{1} := \mathcal S_\mu^0\cap\mathcal C^{1}$ and 
$\mathcal S_{\mu,L}^{1,1} := \mathcal S_\mu^1\cap\mathcal F_L^{1}$.

\subsection{Literature review}\label{sec:review}
We first review some dynamical models involving Hessian data. 
In \cite{Alvarez;Attouch;Bolte;Redont:2002second-order}, combining the well-known continuous Newton method \cite{alvarez_dynamical_1998} and the heavy ball system \cite{polyak_methods_1964}, Alvarez et al. proposed the so-called dynamical inertial Newton (DIN) system
\begin{equation}\label{eq:DIN}
	x'' + \alpha x' + \beta\nabla^2f(x)x' + \nabla f(x) = 0,
\end{equation}
where $\alpha,\beta>0$ are constants and $f\in\mathcal C^2$ is bounded from below. Note that the Hessian term $\nabla^2f(x)x'$ is nothing but the derivative of the gradient $(\nabla f(x))'$. Hence the DIN system \cref{eq:DIN} can be transfered into a first-order system without Hessian
\[
\left\{
\begin{aligned}
y'={}& -(\alpha-1/\beta)x-y/\beta,\\
x'={}& -(\alpha-1/\beta)x-y/\beta-\beta\nabla f(x).
\end{aligned}
\right.
\]
For convex $f$, it has been proved \cite[Theorem 5.1]{Alvarez;Attouch;Bolte;Redont:2002second-order} that each trajectory of \cref{eq:DIN} weakly converges to a minimizer of $f$. Later on, in \cite{attouch_second-order_2012}, Attouch et al. extended the DIN system \cref{eq:DIN} to the composite case $f=h+g$:
\begin{equation*}
	x'' + \alpha x' + \beta\nabla^2h(x)x' + \nabla h(x) +\nabla g(x) = 0,
\end{equation*}
where $g\in \mathcal C^1$ and $f\in\mathcal C^2$ is convex such that $f=h+g$ is convex. 
Like the DIN system \cref{eq:DIN}, this model can also be rewritten as a first-order system
\begin{equation}\label{eq:DIN-attouch-comp}
\left\{
\begin{aligned}
x'={}&-\left(\alpha-1/\beta\right) x-y/\beta -\beta \nabla h(x), \\ 
y'={}&-\left(\alpha-1/\beta\right) x-y/\beta+\beta \nabla g(x),
\end{aligned}
\right.
\end{equation}
based on which they generalized their model to nonsmooth case as well. 

In \cite{attouch_first-order_2019}
Attouch et al. added the Hessian term and time scaling to the ODE derived in \cite{Su;Boyd;Candes:2016differential} and obtained 
\begin{equation}\label{eq:mf-Hessian-exten}
	x''+\frac{\alpha}{t} x'+\beta\nabla^2f(x)x'+b\nabla f(x)=0,
	\quad t\geqslant t_0,
\end{equation}
where $\alpha\geqslant1$ is a constant, $f\in\mathcal C^2$ is convex and $\beta(t)$ is a nonnegative function such that
\[
b(t)>\beta'(t)+\beta(t)/t,
	\quad t\geqslant t_0.
\] 
If $b = 1$, then \cref{eq:mf-Hessian-exten} reduces to the ODE consider in \cite{attouch_fast_2016}.
When $\beta(t)=0$, then \cref{eq:mf-Hessian-exten} coincides with the rescaled ODE derived in \cite{attouch_fast_2018}. When $\alpha=3$, $\beta(t) = \beta>0$ and $b(t) = 1+1.5\beta/t$, then \cref{eq:mf-Hessian-exten} recoveries the high resolution ODE \cref{eq:dan-ODE-2}. They derived the convergence result \cite[Theorem 2.1]{attouch_first-order_2019}
\begin{equation}\label{eq:conv-ode-attouch-mu0}
t^2w(t)(f(x(t))-f(x^*))+\int_{t_0}^{t}s^2\beta(s)
\nm{\nabla f(x(s))}^2\mathrm ds \leqslant C,
\end{equation}
provided that 
\[
w(t) = b(t)-\beta'(t)-\beta(t)/t,
\quad tw'(t)\leqslant (\alpha-3)w(t).
\]
However, due to the above restriction on $w$,
 we have $w(t)\leqslant Ct^{\alpha-3}$ and the best decay rate they can obtain is $O(t^{1-\alpha})$. In \cite{attouch_first-order_2019}, they also studied a DIN system for $f\in\mathcal C^2\cap\mathcal S_\mu^1(\mu>0)$:
\begin{equation}
\label{eq:ode-attouch-mu>0}
	x'' +2\sqrt{\mu}x'  + \beta\nabla f^2(x)x' +\nabla f(x) = 0,
\end{equation}
where $\beta\geqslant 0$ is a constant. Note that the case $\beta=0$ has been considered in \cite{Siegel:2019,Wilson:2018}. For $\beta>0$, they established the result
	\begin{equation}\label{eq:conv-ode-attouch-mu>0}
f(x(t))-f(x^*)+	\beta^2\int_{0}^{t}e^{\sqrt{\mu}(s-t)}
\nm{\nabla f(x(s))}^2
\!\dd s\leqslant C e^{-t\sqrt{\mu}/2}.
\end{equation}

Recently, Shi et al. \cite{Shi:2018} derived two Hessian-driven models, which were called high-resolution ODEs. One requires $f\in\mathcal C^2\cap\mathcal S_{\mu,L}^{1,1}$ with $\mu>0$ and reads as follows
\begin{equation}
\label{eq:ode-shi-mu>0}
	x'' +2\sqrt{\mu}x'  + \sqrt{\beta}\nabla f^2(x)x' 
	+ (1+\sqrt{\mu \beta})\nabla f(x) = 0,
\end{equation}
where $0<\beta\leqslant 1/L$. This ODE interprets \cite[Constant step scheme, III, Chapter 2]{Nesterov:2013Introductory} and achieves the exponential decay \cite[Theorem 1 and Lemma 3.1]{Shi:2018}
	\begin{equation}\label{eq:conv-ode-shi-mu>0}
f(x(t))-f(x^*)+	\sqrt{\beta}
\int_{0}^{t}e^{(s-t)\sqrt{\mu}/4}
\nm{\nabla f(x(s))}^2
\!\dd s\leqslant C e^{-t\sqrt{\mu}/4}.
\end{equation}
The second is for $f\in\mathcal C^2\cap \mathcal F_L^{1}$:
\begin{equation}\label{eq:dan-ODE-2}
	x'' +\frac{3}{t} x'+\sqrt{\beta} \nabla^{2} f(x) x'+
	\left(1+t_0/t\right) \nabla f(x)=0,\quad t
	\geqslant t_0 = 1.5\sqrt{\beta},
\end{equation}
where $\beta>0$. This model agrees with \cref{eq:mf-Hessian-exten} in a special case that  $\alpha=3$, $\beta(t) = \sqrt{\beta}$ and $b(t) = 1+1.5\sqrt{\beta}/t$, and the convergence result \cref{eq:conv-ode-attouch-mu0} in this case has also been proved by \cite[Lemma 4.1 and Corollary 4.2]{Shi:2018}. Compared with the dynamical systems derived in \cite{attouch_first-order_2019,Shi:2018}, our H-NAG flow \cref{eq:HNAG} uniformly treats $f\in\mathcal S_\mu^1$ with $\mu\geqslant 0$ and yields the convergence result \cref{eq:conv-Lt-HNAG} which, also gives the estimate for the gradient as what \cref{eq:conv-ode-attouch-mu0}, \cref{eq:conv-ode-attouch-mu>0} and \cref{eq:conv-ode-shi-mu>0} do. 

Optimization methods based on differential equation solvers for those systems above are also proposed. Based on a semi-implicit scheme for \cref{eq:DIN-attouch-comp}, Attouch et al. \cite{attouch_dynamical_2014} proposed an inertial forward-backward algorithm for composite convex optimization and established the weak convergence. In \cite{castera_inertial_2019}, Castera et al. applied the DIN system \cref{eq:DIN} to deep neural networks and presented an inertial Newton algorithm for minimizing the empirical risk loss function. Their numerical experiments showed that the proposed method performs much better than SGD and Adam in the long run and can reach very low training error. With minor change of \cref{eq:dan-ODE-2}, Shi et al. \cite{Shi:2018} developed a family of accelerated methods by explicit discretization scheme. Later in \cite{shi_acceleration_nodate} , for \cref{eq:ode-shi-mu>0}, they considered explicit and symplectic methods, among which only the symplectic scheme achieves the accelerated rate \cref{eq:opt-rate}. More recently, Attouch et al. \cite{attouch_first-order_2019} proposed two explicit schemes for \cref{eq:mf-Hessian-exten,eq:ode-attouch-mu>0}, respectively. However, only the discretization for \cref{eq:mf-Hessian-exten} has accelerated rate $O(1/k^2)$; see \cite[Theorem 3.3]{attouch_first-order_2019}. We emphasize that, our \cref{algo:HNAG,algo:HNAG-comp} possess the convergence rate
\begin{equation}\label{eq:opt-rate}
O\left(
\min
\left\{
1/k^2,\,\big(1+\sqrt{\mu/L}\big)^{-k}
\right\}
\right),
\end{equation}
which is optimal  for $f\in\mathcal S_\mu^1(\mu\geqslant 0)$ and 
accelerated for $f\in\mathcal S_\mu^0(\mu\geqslant 0)$. More methods that achieve the rate \cref{eq:opt-rate} are listed in Sections \ref{sec:HNAG-gd}, \ref{sec:methodfromNAG} and \ref{sec:HNAG-gm}.

The rest of this paper is organized as follows. 
In Section 
\ref{sec:conti} we
focus on the continuous problem \cref{eq:HNAG}. 
Then, in Sections \ref{sec:semi-im} and \ref{sec:explicit}
we consider (semi-)implicit and explicit schemes sequentially.
Then, we deal with the composite case $f=h+g$ in 
Section \ref{sec:split}. Finally, we give conclusion and future work in Section \ref{sec:con}.

\section{Continuous Problem}
\label{sec:conti}
In this section, we study our H-NAG flow for $f\in\mathcal S_\mu^1$($\mu\geqslant 0$) and establish the minimizing property of the trajectory. 
\subsection{Notation}
To move on and for later use, throughout this paper, we define the Lyapunov function $\mathcal L:
\textbf{\emph{V}}\to\mathbb R_{\geqslant 0}$ by that
\begin{equation}\label{eq:L}
\mathcal L(\bs x) =\mathcal L(x,v,\gamma) := f(x) - f(x^*) + \frac{\gamma}{2}\nm{v-x^*}^2.
\end{equation}
where $\bs x = (x,v,\gamma)\in  V\times V\times \mathbb R_+:=\textbf{\emph{V}}$, and $x^*$ is a global minimum point of $f$. 
When $\bs x(t)=(x(t),v(t),\gamma(t))$ is a $\textbf{\emph{V}}$-valued function of time variable $t$ on $[0,\infty)$, we also introduce the abbreviated notation
\begin{equation}\label{eq:Lt}
	\mathcal L(t) :=  \mathcal L\big(\bs x(t)\big)
	=\mathcal L\big(x(t),v(t),\gamma(t)\big),\quad t\geqslant 0.
\end{equation}
Note that $\mathcal L$ is convex with respect to $(x,v)$ and linear in $\gamma$. 
Moreover, $\mathcal L$ is whenever smooth in respect of $(v,\gamma)$ and it is trivial that
\begin{align*}
	\nabla_x \mathcal L & = \nabla f(x),\\
	\nabla_v \mathcal L & = \gamma ( v - x^*),\\
	\nabla_{\gamma} \mathcal L & = \frac{1}{2}\nm{v-x^*}^2.
\end{align*}
Above, $\nabla_{\!\times }$ means the partial derivative of $\times = x,\,v$ or $\gamma$. For any $\beta \in \mathbb R_{+}$ and $\bs x=(x,v,\gamma)\in \textbf{\emph{V}}$, we introduce the flow field $\mathcal G$
\begin{equation}\label{eq:G}
\mathcal G(\bs x,\beta):=
\big(\mathcal G^x(\bs x,\beta),\, \mathcal G^v(\bs x), \,\mathcal G^{\gamma}(\bs x)\big),
\end{equation}
where the three components are defined as follows
\begin{align*}
&	\mathcal G^x(\bs x,\beta)  = v - x - \beta \nabla f(x),\\
&	\mathcal G^v(\bs x)  = \frac{\mu}{\gamma}(x - v) - \frac{1}{\gamma} \nabla f(x),\\
&	\mathcal G^{\gamma} (\bs x) = \mu - \gamma.
\end{align*}

Our H-NAG system \cref{eq:HNAG} can be simply written as 
\begin{equation}\label{eq:shortHNAG}
	\bs x'(t) = \mathcal G(\bs x(t),\beta(t)),
\end{equation}
where $\bs x(t)=(x(t),v(t),\gamma(t))$. We find that $\bs x^*=(x^*, x^*, \mu)$ is a candidate of the equilibrium point to the dynamic system \cref{eq:shortHNAG}. 

The well-posedness of \cref{eq:shortHNAG} is standard. Indeed, if $f$ has Lipschitz continuous gradient, then apply the classical existence and uniqueness results of 
ODE (see~\cite[Theorem 4.1.4]{Ahmad2015}) yields that the ODE system \cref{eq:shortHNAG} admits a unique solution $\bs x =(x,v,\gamma)$ 
with $x\in\mathcal C^2([0,\infty);V)$ and $v\in\mathcal C^1([0,\infty);V)$.

\subsection{Strong Lyapunov property}
Originally the Lyapunov function is used to study the stability of an equilibrium point of a dynamical system. The function $\mathcal L(\bs x)$ defined by \cref{eq:L} is called a Lyapunov function of the vector field $\mathcal G(\bs x,\beta)$ \cref{eq:G} near an equilibrium point $\bs x^*$  if $\mathcal L(\bs x^*) = 0$ and
\begin{equation}\label{eq:Ly-cond}
-\nabla \mathcal L(\bs x)\cdot 
\mathcal G(\bs x,\beta)\text{ is locally positive near } \bs x^*.
\end{equation}
 
To obtain the convergence rate, we need a stronger condition than merely $-\nabla \mathcal L(\bs x)\cdot \mathcal G(\bs x,\beta)$ is locally positive definite. We introduce the strong Lyapunov property: there exist a positive function $c(\bs x)> 0$, and a function $q(\bs x):\textbf{\emph{V}}\to\mathbb R$ such that 
\begin{equation}\label{eq:A}
	-\nabla \mathcal L(\bs x)\cdot \mathcal G(\bs x,\beta)
	\geqslant
	c(\bs x)\mathcal L(\bs x)+q^2(\bs x)
	\quad \forall\,\bs x\in \textbf{\emph{V}}.
\end{equation}

Next we will show the Lyapunov function \eqref{eq:L} satisfies the strong Lyapunov property. 
\begin{lem}\label{lem:LG}
Assume $f\in\mathcal S_\mu^1(\mu\geqslant 0)$. For any $\beta \in \mathbb R_{+}$ and $\bs x =(x,v,\gamma)\in \textbf{V}$, we have
	\begin{equation}\label{eq:LG}
		-\nabla \mathcal L(\bs x)\cdot\mathcal G(\bs x,\beta)
		\geqslant \mathcal L(\bs x)+ \beta \|\nabla f(x)\|^2 + \frac{\mu}{2}\nm{x-v}^2.
	\end{equation}
\end{lem}
\begin{proof}
	Indeed, observing the identity 
	\[
	2\dual{x-v,v-x^*}=\nm{x-x^*}^2-\nm{x-v}^2-\nm{v-x^*}^2,
	\]
	and using the convexity of $f$
	$$
	\dual{\nabla f(x),x-x^*} \geqslant f(x) - f(x^*) + \frac{\mu}{2} \nm{x-x^*}^2,
	$$
	a direct computation gives 
	\begin{equation*}
		\begin{split}
			-\nabla \mathcal L(\bs x) \cdot \mathcal G(\bs x,\beta) ={}& - \mu\dual{x-v,v-x^*}+\dual{\nabla f(x),x-x^*} \\
			{}&\qquad+ \beta\nm{\nabla f(x)}^2+ \frac{\gamma - \mu}{2}\nm{v-x^*}^2\\
			\geqslant{} & \mathcal L(\bs x)+ \beta \nm{\nabla f(x)}^2 + \frac{\mu}{2}\nm{x-v}^2.
		\end{split}
	\end{equation*} 
	This finishes the proof of this lemma.
\end{proof}

When $f$ is nonsmooth, we introduce the notation $\partial \mathcal L(\bs x, p) =\big(p,\nabla_v \mathcal L(\bs x),\nabla_\gamma \mathcal L(\bs x)\big)$ and $\mathcal G(\bs x, \beta,p)$ by replacing $\nabla f(x)$ in $\mathcal G$ with some $p\in\partial f(x)$. Namely we substitute $\nabla f(x)$ in $\nabla _x\mathcal L$ and $\mathcal G$ with some $p\in\partial f(x)$, where the sub-gradients $\partial f(x)$ of $f$ is defined in \cref{eq:sub}.
Then we can easily generalize \cref{lem:LG} to the non-smooth version. 

\begin{lem}\label{lem:LG-non}
Assume $f\in\mathcal S_\mu^0(\mu\geqslant 0)$. 
Then for any $\beta \in \mathbb R_{+},\,\bs x=(x,v,\gamma)\in \textbf{V}$ and $p\in \partial f(x)$, we have
	\begin{equation}\label{eq:LG-non}
		-\partial \mathcal L(\bs x, p) 
		\cdot \mathcal G(\bs x,\beta, p)
		\geqslant \mathcal L(\bs x)+ \beta \| p \|^2 + \frac{\mu}{2}\nm{x-v}^2.
	\end{equation}
\end{lem}
\subsection{Minimizing property}
The crucial inequality \cref{eq:LG} implies that $\mathcal G$ is a descent direction for minimizing $\mathcal L$ and thus $\mathcal L$ and $\nm{\nabla f}$ decrease along the trajectory defined by \cref{eq:shortHNAG}. Indeed, we have the following theorem that depicts this.
\begin{thm}
	Let $\bs x(t)=(x(t), v(t), \gamma(t))$ be the solution of \cref{eq:shortHNAG},
	then for any $t\geqslant 0$,
	\begin{equation}\label{eq:conv-Lt}
		\mathcal L(t)+	\int_{0}^{t}e^{s-t}\beta(s)
		\nm{\nabla f(x(s))}^2
		\!\dd s\leqslant e^{-t}\mathcal L(0).
	\end{equation}
\end{thm}
\begin{proof}
	By the chain rule $\mathcal L'(t) = \nabla \mathcal L(\bs x (t))\cdot \mathcal G(\bs x(t),\beta(t))$ and the key estimate \cref{eq:LG}, we have the inequality
	\[
	\mathcal L'(t)\leqslant-\mathcal L(t) 
	-\beta(t)\nm{\nabla f(x(t))}^2-\frac{\mu}{2}\nm{x(t)-v(t)}^2
	\leqslant-\mathcal L(t).
	\]
	This yields the exponential decay rate $\mathcal L(t)\leqslant e^{-t}\mathcal L(0)$. Moreover, we find that 
	\[
	\mathcal L'(t) + \mathcal L(t) + 
	\beta(t)\nm{\nabla f(x(t))}^2
	\leqslant0.
	\]
	Multiplying both sides by $e^{t}$
	and integrating over $(0,t)$ gives 
	\begin{equation*}
		\int_{0}^{t}\dd \left(e^s\mathcal L(s)\right)+
		\int_{0}^{t}e^{s}\beta(s)
		\nm{\nabla f(x(s))}^2
		\!\dd s
		\leqslant 0,
	\end{equation*}
	which also implies 
	\begin{equation*}
		e^t\mathcal L(t)+
		\int_{0}^{t}e^{s}\beta(s)
		\nm{\nabla f(x(s))}^2
		\!\dd s
		\leqslant 
		\mathcal L(0)
		,\quad t\geqslant0.
	\end{equation*}
	This proves \cref{eq:conv-Lt} and establishes the proof of this theorem.
\end{proof}
\begin{rem}
We do not have to give the explicit form of $\beta(t)$, which is acceptable as long as it is positive, i.e., $\beta(t)>0$ for all $t>0$. In the discretization level, however, to obtain optimal rate of convergence, we shall choose special coefficient $\beta_k$, which is positive and computable (cf. \cref{thm:conv-semi} and \cref{thm:conv-ex1-ode-NAG}).
\end{rem}

\begin{rem}
As discussed in \cite[section 2.2]{luo_chen_2019_from}, the exponential decay \cref{eq:conv-Lt} may be sped or slowed down if we introduce the time rescaling. In our model \cref{eq:HNAG}, such rescaling is automatically encoded in the damping 
parameter $\gamma$ governed by the equation 
$\gamma'=\mu-\gamma$ which allow us to handle $\mu> 0$ and $\mu =0$ in a unified way. 
\end{rem}
\section{A Semi-implicit Scheme}
\label{sec:semi-im}
In this section, we consider a semi-implicit scheme for our H-NAG flow \cref{eq:HNAG}, where $f\in\mathcal S_\mu^1$ with $\mu\geqslant 0$. 
We will see that in the discrete level, rescaling effect and exponential decay can be inherit by (semi-)implicit scheme which has no restriction on step size; see \cref{thm:conv-semi}, \cite[Theorem 3.1]{attouch_fast_2018} and \cite[Theorem 1]{luo_chen_2019_from}. 

Our scheme reads as follows
\begin{equation}\label{eq:semi-HNAG}
	\left\{
	\begin{aligned}
		\frac{x_{k+1}-x_{k}}{\alpha_k}={}& v_{k}-x_{k+1}-\beta_k\nabla f(x_{k+1}),\\
		\frac{v_{k+1}-v_{k}}{\alpha_k}={}&
		\frac{\mu }{\gamma_k}(x_{k+1}-v_{k+1})
		-\frac{1}{\gamma_k}\nabla f(x_{k+1}),\\
		\frac{\gamma_{k+1} - \gamma_{k} }{\alpha_k}  ={}& 
		\mu -\gamma_{k+1}.
	\end{aligned}
	\right.
\end{equation}
If we set 
\[
y_k:=\frac{x_k+\alpha_kv_k}{1+\alpha_k} ,\quad 
s_k:=\frac{\alpha_k\beta_k}{1+\alpha_k},
\]
then the update for $x_{k+1}$ is equivalent to 
\[
x_{k+1} = y_k-s_k\nabla f(x_{k+1})=\proxi_{s_k f}(y_k).
\]
After obtaining $x_{k+1}$, $v_{k+1}$ is obtained through the second equation of \cref{eq:semi-HNAG}.

To characterize the convergence rate, denote by
\begin{equation}\label{eq:lambdak}
	\lambda_0=1,\quad\lambda_k = \prod_{i=0}^{k-1}\frac{1}{1+\alpha_i},\quad k\geqslant 1.
\end{equation}
We introduce the discrete Lyapunov function
\begin{equation}
	\label{eq:Lk}
	\mathcal L_k
	:= \mathcal L(\bs x_k)
	= {}f(x_k)-f(x^*) +
	\frac{\gamma_k}{2}
	\nm{v_k-x^*}^2,
\end{equation}
where $\bs x_k=(x_k, v_k, \gamma_k)$, and 
\begin{equation}\label{eq:Rk-semi}
	\mathcal R_0=0,\quad
	\mathcal R_k: =\frac{\lambda_k}{2}\sum_{i=0}^{k-1}
	\frac{\alpha_{i}\beta_i}{\lambda_i}\nm{\nabla f(x_{i+1})}^2,\quad k\geqslant 1.
\end{equation}
Furthermore, for all $k\geqslant 0$, we set
\begin{equation}\label{eq:Ek-semi}
	\mathcal E_k=\mathcal L_k+\mathcal R_k.
\end{equation}
In the following, we present the convergence result 
for our semi-implicit scheme~\cref{eq:semi-HNAG}. 
\begin{thm}\label{thm:conv-semi}
	Assume $\beta_k$ satisfies $ \beta_k \gamma_k=\alpha_k $, then for the semi-implicit scheme~\cref{eq:semi-HNAG} with any step size $\alpha_k>0$, we have
	\begin{equation}\label{eq:diff-Ek-semi}
		\mathcal E_{k+1}\leqslant 
		\frac{	\mathcal E_k}{1+\alpha_k }\quad\forall\,k\geqslant 0.
	\end{equation}
	Consequently, for all $k\geqslant 0$, it holds that
	\begin{equation}\label{eq:conv-im-g}
		\mathcal L_k+\frac{\lambda_k}{2}\sum_{i=0}^{k-1}
		\frac{\alpha_{i}^2}{\lambda_i\gamma_i}
		\nm{\nabla f(x_{i+1})}^2
		\leqslant\lambda_k\mathcal L_0.
	\end{equation}
\end{thm}
\begin{proof}
	We first split the difference as
	\begin{align*}
		\mathcal L_{k+1}- \mathcal L_k={}&
		\mathcal L(x_{k+1}, v_{k},\gamma_{k}) - \mathcal L(x_{k}, v_{k},\gamma_{k})\\
		{}&+\mathcal L(x_{k+1}, v_{k+1},\gamma_{k}) - \mathcal L(x_{k+1}, v_{k},\gamma_{k}) \\
		&+\mathcal L(x_{k+1}, v_{k+1},\gamma_{k+1}) - \mathcal L(x_{k+1}, v_{k+1},\gamma_{k}) \\
		:={}& {\rm I}_1 + {\rm I}_2 + {\rm I}_3.		
	\end{align*}
	The last item ${\rm I}_3$ is the easiest one as $\mathcal L$ is linear in $\gamma$
	\begin{equation}\label{eq:I3}
		{\rm I}_3 = \dual{\nabla_{\gamma}\mathcal L(\bs x_{k+1}), \gamma_{k+1} - \gamma_{k}} = \alpha_k (\nabla_{\gamma}\mathcal L(\bs x_{k+1}), \mathcal G^{\gamma}(\bs x_{k+1})).
	\end{equation}
	For item ${\rm I}_2$, we use the fact $\mathcal L(x_{k+1}, \cdot, \gamma_k)$ is $\gamma_k$-convex and the discretization \cref{eq:semi-HNAG} to get
	\begin{align}
		{\rm I}_2\leqslant {} &\dual{\nabla_v \mathcal L( x_{k+1}, v_{k+1}, \gamma_{k}), v_{k+1} - v_k} - \frac{\gamma_k}{2}\nm{ v_{k+1} - v_k}^2\nonumber\\
		={}&\alpha_k \dual{\nabla_v \mathcal L(\bs x_{k+1}), \mathcal G^v(\bs x_{k+1})}- \frac{\gamma_k}{2}\nm{ v_{k+1} - v_k}^2.\label{eq:I2}
	\end{align}
	In the last step, as $\gamma_k$ is canceled in the product, we can switch the argument $\gamma_k$ to $\gamma_{k+1}$. 
	By the convexity of $f$, it is clear that 
	\begin{equation*}
		\begin{split}
			{}&{\rm I}_{1} =f(x_{k+1})-f(x_k)\leqslant \dual{\nabla f(x_{k+1}), x_{k+1} - x_k}\\
			={}&\alpha_k\dual{\nabla_{x}\mathcal L(\bs x_{k+1}), \mathcal G^x(\bs x_{k+1},\beta_k)}+
			\alpha_k\dual{\nabla f(x_{k+1}), v_k-v_{k+1}}.
		\end{split}
	\end{equation*}
	Observing the negative term in \cref{eq:I2}, we bound the second term as follows
	$$
	\alpha_k \| \nabla f(x_{k+1})\|\| v_k - v_{k+1}\|
	\leqslant \frac{\alpha_k^2}{2\gamma_k} 
	\| \nabla f(x_{k+1})\|^2 + \frac{\gamma_k}{2}\| v_k - v_{k+1}\|^2.
	$$
	Now, adding all together and using the strong Lyapunov property  \cref{eq:LG}, we get 
	\begin{align}
		\mathcal L_{k+1}- \mathcal L_k
		\leqslant {}&\alpha_k \big(\nabla \mathcal L(\bs x_{k+1}), \mathcal G(\bs x_{k+1},\beta_k) \big) + \frac{\alpha_k^2}{2\gamma_k}  \| \nabla f(x_{k+1})\|^2\nonumber\\
		\leqslant {}&- \alpha_k \mathcal L_{k+1}+ \left (\frac{\alpha_k^2}{2\gamma_k} -\alpha_k\beta_k\right ) \|\nabla f(x_{k+1})\|^2\nonumber\\
		= {}&- \alpha_k \mathcal L_{k+1} -\frac{\alpha_k\beta_k}{2} \|\nabla f(x_{k+1})\|^2.		\label{eq:diff-Lk-im}
	\end{align}
	Finally, by definition $\lambda_{k+1}-\lambda_k=-\alpha_k\lambda_{k+1}$, it is evident that 
	\begin{align}
		2\mathcal R_{k+1}-2\mathcal R_k={}&
		\lambda_{k+1}\sum_{i=0}^{k}
		\frac{\alpha_{i}\beta_i}{\lambda_i}\nm{\nabla f(x_{i+1})}^2
		-\lambda_k\sum_{i=0}^{k-1}
		\frac{\alpha_{i}\beta_i}{\lambda_i}\nm{\nabla f(x_{i+1})}^2\nonumber\\
		={}&\alpha_k\beta_k
		\nm{\nabla f(x_{k+1})}^2+
		(\lambda_{k+1}-\lambda_{k})
		\sum_{i=0}^{k}
		\frac{\alpha_{i}\beta_i}{\lambda_i}\nm{\nabla f(x_{i+1})}^2\nonumber\\
		={}&\alpha_k\beta_k
		\nm{\nabla f(x_{k+1})}^2-\alpha_k2\mathcal R_{k+1}.\label{eq:diff-Rk-im}
	\end{align}
	Now combining the relation $ \beta_k \gamma_k=\alpha_k $ with \cref{eq:diff-Lk-im,eq:diff-Rk-im} implies \cref{eq:diff-Ek-semi} 
	and thus concludes the proof of this theorem.
\end{proof}

With carefully designed parameter $\beta_k=\alpha_k/\gamma_k$, the semi-implicit scheme~\cref{eq:semi-HNAG} 
can always achieve linear convergence rate as long as the step size $\alpha_k$ is chosen uniformly bounded below
$\alpha_k\geqslant\widehat{\alpha}>0$ for all $k>0$ and 
larger $\alpha_k$ yields faster convergence rate. Observing the update of $\gamma_{k+1}$, we conclude that, if $\gamma_0\geqslant \mu$, then $\gamma_k\geqslant \gamma_{k+1}\geqslant \mu$, and if $0< \gamma_0<\mu$, then $\gamma_k< \gamma_{k+1}<\mu$. Hence, it follows that
\begin{equation}\label{eq:low-gk}
	\min\{\gamma_0,\mu\}\leqslant \gamma_k\leqslant \max\{\gamma_0,\mu\},
\end{equation}
and 
from \cref{eq:conv-im-g} we can get fast convergence for the norm of the gradient.
\begin{rem}
	If $f$ is nonsmooth, we use the proximal operator $\proxi_{s_k f}$ to rewrite the implicit scheme \cref{eq:semi-HNAG} as follows
	\begin{equation*}
		\left\{
		\begin{aligned}
			x_{k+1} = {}&\proxi_{s_k f}(y_k), \quad y_k=\frac{x_k+\alpha_kv_k}{1+\alpha_k},\quad s_k=\frac{\alpha_k\beta_k}{1+\alpha_k}, \\
			p_{k+1}={}&\frac{1}{\beta_k}\left(v_{k}-x_{k+1} -\frac{x_{k+1}-x_{k}}{\alpha_k}\right),\\
			\frac{v_{k+1}-v_{k}}{\alpha_k}={}&
			\frac{\mu }{\gamma_k}(x_{k+1}-v_{k+1})
			-\frac{1}{\gamma_k}p_{k+1},\\
			\frac{\gamma_{k+1} - \gamma_{k} }{\alpha_k}  ={}& 
			\mu -\gamma_{k+1}.
		\end{aligned}
		\right.
	\end{equation*}
	Note that $p_{k+1}\in \partial f(x_{k+1})$. We just simply replace $\nabla f(x_{k+1})$ by $p_{k+1}$. 	In addition, thanks to \cref{lem:LG-non}, proceeding as the proof of 
	\cref{thm:conv-semi}, we can derive 
	\begin{equation*}
		\mathcal L_k+\frac{\lambda_k}{2}\sum_{i=0}^{k-1}
		\frac{\alpha_{i}\beta_i}{\lambda_i}\nm{p_{i+1}}^2
		\leqslant\lambda_k\mathcal L_0.
	\end{equation*}
\end{rem}
\section{Explicit Schemes with Optimal Rates}
\label{sec:explicit}
This section assumes $f\in\mathcal S_{\mu,L}^{1,1}$ 
with $\mu\geqslant 0$ and considers several explicit 
schemes including \cref{algo:HNAG}. All of those methods have optimal convergence rates in the sense of Nesterov \cite[Chapter 2]{Nesterov:2013Introductory}.
\subsection{Analysis of \cref{algo:HNAG}}
It is straightforward to verify that the \cref{algo:HNAG} is equivalent to the following explicit scheme 
\begin{equation}\label{eq:ex-HNAG}
	\left\{
	\begin{aligned}
		\frac{x_{k+1}-x_{k}}{\alpha_k}={}& v_{k}-x_{k+1}-\beta_k\nabla f(x_{k}),\\
		\frac{v_{k+1}-v_{k}}{\alpha_k}={}&
		\frac{\mu }{\gamma_k}(x_{k+1}-v_{k+1})
		-\frac{1}{\gamma_k}\nabla f(x_{k+1}),\\
		\frac{\gamma_{k+1} - \gamma_{k} }{\alpha_k}  ={}& 
		\mu -\gamma_{k+1},
	\end{aligned}
	\right.
\end{equation}
where 
\begin{equation}\label{eq:ab}
	\alpha_k=\sqrt{\frac{\gamma_k}{L}},\quad \beta_k = \frac{1}{L\alpha_k}.
\end{equation}
Given $(x_k, v_k, \gamma_k)$, we can solve the first equation to get $x_{k+1}$ and with known $x_{k+1}$, we can get $v_{k+1}$ from the second equation. 
Moreover, the sequence $\{v_k\}$ can be further eliminated to get an equation of $(x_{k+1}, x_k, x_{k-1})$
\begin{equation*}
	\begin{aligned}
		\displaystyle
		{}&\gamma_k\cdot\frac{\frac{x_{k+1}-x_k}{\alpha_k}
			-\frac{x_{k}-x_{k-1}}{\alpha_{k-1}}}{\alpha_k}
		+(\mu+\gamma_k)\cdot
		\frac{x_{k+1}-x_k}{\alpha_k}\\
		&\qquad+\gamma_k\beta_k\cdot
		\frac{\nabla f(x_{k})-\nabla f(x_{k-1})}{\alpha_k}
		+(1+\mu\beta_{k})\nabla f(x_{k})\\
		{}&\qquad\qquad+\gamma_k\cdot\frac{\beta_k-\beta_{k-1}}
		{\alpha_k}\cdot\nabla f(x_{k-1}) = 0,
	\end{aligned}
\end{equation*}
which is an explicit scheme for \cref{eq:ode-agf-H} since the unknown $x_{k+1}$ is not in the gradient. Note that Hessian term $\nabla^2 f$ is not present as the action $\nabla^2 f(x)x'$ can be discretized by the quotient of the gradient. 

For the convergence analysis, we need the following tighter bound on the function difference; see \cite[Theorem 2.1.5]{Nesterov:2013Introductory}.
\begin{lem}[\cite{Nesterov:2013Introductory}]
	\label{lem:key-est}
	If $ f\in\mathcal F_L^1$, then
	\begin{equation*}
		f(y)-f(x) \leqslant 
		\dual{\nabla f(y), y-x} 
		- \frac{1}{2L}\nm{\nabla f( y) - \nabla  f(x)}^2\quad\forall\, x,y\in V.
	\end{equation*}
\end{lem}

For the explicit scheme, we modify the definition \cref{eq:Rk-semi} 
of $\mathcal R_k$ slightly as
\begin{equation*}
	\mathcal R_0=0,\quad
	\mathcal R_k: =\frac{\lambda_k}{2}\sum_{i=0}^{k-1}
	\frac{\alpha_i\beta_i}{\lambda_i}\nm{\nabla f(x_{i})}^2,\quad k\geqslant 1,
\end{equation*}
and we also set $\mathcal E_k:=\mathcal L_k+\mathcal R_k$, where $\lambda_k$ and $\mathcal L_k$ are defined in \cref{eq:lambdak,eq:Lk}, respectively. Similar to the derivation of \cref{eq:diff-Rk-im}, we have
\begin{equation}\label{eq:diff-Rk-ex}
	\mathcal R_{k+1}-\mathcal R_k ={}-\alpha_k\mathcal R_{k+1} +\frac{\alpha_k\beta_k}{2}
	\nm{\nabla f(x_{k})}^2.
\end{equation}

\begin{thm}
	\label{thm:conv-ex1-ode-NAG}
	For \cref{algo:HNAG}, we have
	\begin{equation}\label{eq:conv1-ex1-ode-NAG}
		\mathcal E_{k+1}
		\leqslant 
		\frac{	 \mathcal E_k}{1+\alpha_k }\quad\forall\,k\geqslant 0.
	\end{equation}
	Consequently, for all $k\geqslant0$, it holds that
	\begin{equation}\label{eq:conv-ex}
		\mathcal L_k+\frac{1}{2L}\sum_{i=0}^{k-1}
		\frac{\lambda_k}{\lambda_i} \nm{\nabla f(x_{i})}^2
		\leqslant\lambda_k\mathcal L_0.
	\end{equation}
	Above, $\lambda_k$ is bounded above by the optimal convergence rate
	\begin{equation}\label{eq:lambdak-ex}
		\lambda_k\leqslant \min\left\{
		8L\left(2\sqrt{2L}+\sqrt{\gamma_0}k\right)^{-2},\,
		\left(1+\sqrt{\min\{\gamma_0,\mu\}/L}\right)^{-k}
		\right\}.
	\end{equation}
\end{thm}
\begin{proof}
	Following the proof of \cref{thm:conv-semi}, we first split the difference $		\mathcal L_{k+1} - \mathcal L_k$ along the path $\bs x_k=(x_k, v_{k},\gamma_{k})$ to $(x_{k+1}, v_{k},\gamma_{k})$ to $(x_{k+1}, v_{k+1},\gamma_{k})$ and finally to  $\bs x_{k+1}=(x_{k+1}, v_{k+1},\gamma_{k+1})$:
	\begin{align*}
		\mathcal L_{k+1} - \mathcal L_k
		={}&  \mathcal L(x_{k+1}, v_{k},\gamma_{k}) - \mathcal L(x_k, v_{k},\gamma_{k}) \\
		&+ \mathcal L(x_{k+1}, v_{k+1},\gamma_{k}) - \mathcal L(x_{k+1}, v_{k},\gamma_{k}) \\
		&+\mathcal L(x_{k+1}, v_{k+1},\gamma_{k+1}) - \mathcal L(x_{k+1}, v_{k+1},\gamma_{k}) \\
		:={}& {\rm I}_1 + {\rm I}_2 + {\rm I}_3.
	\end{align*}
	Note that we still have \cref{eq:I3,eq:I2}:
	\begin{align}
		{\rm I}_3 ={}& \alpha_k (\nabla_{\gamma}\mathcal L(\bs x_{k+1}), \mathcal G^{\gamma}(\bs x_{k+1})),\nonumber\\
		{\rm I}_2 \leqslant{}&  \alpha_k \dual{\nabla_v \mathcal L(\bs x_{k+1}), \mathcal G^v(\bs x_{k+1})} - \frac{\gamma_k}{2}\nm{ v_{k+1} - v_k}^2.
			\label{eq:I2-ex}
	\end{align}
	We now use \cref{lem:key-est} to estimate ${\rm I}_1$
	\begin{equation*}
		\begin{split}
			{\rm I}_1
			\leqslant {}&\dual{\nabla_x \mathcal L(\bs x_{k+1}), x_{k+1} - x_k} - \frac{1}{2L}\nm{\nabla f(x_{k+1}) - \nabla f(x_k)}^2.
		\end{split}
	\end{equation*}
	In the first step, we can switch $(x_{k+1}, v_k,\gamma_k)$ to $\bs x_{k+1}$ because $\nabla_x \mathcal L$ is independent of $(v,\gamma)$. Then we use the discretization \cref{eq:ex-HNAG} to replace $x_{k+1} - x_k$ and compare with the flow evaluated at $\bs x_{k+1}$:
	\begin{align*}
		\dual{\nabla_x \mathcal L(\bs x_{k+1}), x_{k+1} - x_k} 
		={}&  \alpha_k \dual{\nabla_x \mathcal L(\bs x_{k+1}), \mathcal G^x(\bs x_{k+1},\beta_k)} \\
		& +\alpha_k\beta_k ( \nabla f(x_{k+1}) , \nabla f(x_{k+1}) - \nabla f(x_k))\\
		& +\alpha_k \dual{\nabla f(x_{k+1}), v_k - v_{k+1}}.
	\end{align*}
	Observing the bound \cref{eq:I2-ex} for ${\rm I}_2$, we  use Cauchy\textendash Schwarz inequality to bound the last term as follows
	\begin{equation}\label{eq:df-diff-vk}
		\begin{split}
			\alpha_k \| \nabla f(x_{k+1})\|\| v_k - v_{k+1}\|\leqslant {}&
			\frac{\alpha_k^2}{2\gamma_k} \| \nabla f(x_{k+1})\|^2 
			+ \frac{\gamma_k}{2}\| v_k - v_{k+1}\|^2.
		\end{split}
	\end{equation}
	We use the identity for the second term
	\begin{align*}
		&\alpha_k\beta_k( \nabla f(x_{k+1}) , \nabla f(x_{k+1}) - \nabla f(x_k)) \\
		= & - \frac{\alpha_k\beta_k}{2}\| \nabla f(x_{k})\|^2
		+ \frac{\alpha_k\beta_k}{2}\| \nabla f(x_{k+1})\|^2 + \frac{\alpha_k\beta_k}{2}\| \nabla f(x_{k+1}) - \nabla f(x_k) \|^2.
	\end{align*}
	Adding all together and applying \cref{lem:LG} yield that 
	\begin{equation}\label{diff-Lk}
		\begin{split}
			\mathcal L_{k+1} - \mathcal L_k
			\leqslant {}& -\alpha_k\mathcal L_{k+1}  -\frac{\alpha_k\beta_k}{2}\nm{\nabla f(x_k)}^2  \\
			& +\frac{1}{2}\left ( \alpha_k\beta_k - \frac{1}{L}\right )\| \nabla f(x_{k+1}) - \nabla f(x_k) \|^2\\
			& +\frac{1}{2}\left ( \frac{\alpha_k^2}{\gamma_k}  - \alpha_k\beta_k \right )\nm{\nabla f(x_{k+1})}^2.
		\end{split}
	\end{equation}
	Additionally, in view of the choice of parameters $\alpha_k$ and $ \beta_k$ (cf. \cref{eq:ab}), we have
	\[
	\alpha_k\beta_k-\frac{1}{L}=0,\quad
	\frac{\alpha_k^2}{\gamma_k}  
	-\alpha_k\beta_k = 0,
	\]
	which implies 
	\begin{align*}
		\mathcal L_{k+1} - \mathcal L_k
		\leqslant {}&- \alpha_k \mathcal L_{k+1}
		-\frac{\alpha_k\beta_k}{2}\nm{\nabla f(x_k)}^2.
	\end{align*}
	This  together \cref{eq:ab,eq:diff-Rk-ex} gives the desired estimates \cref{eq:conv1-ex1-ode-NAG,eq:conv-ex}.
	
	Next, let us study the asymptotic behavior of $\lambda_k$. 
	The formula of $\gamma_k$ yields
	\[
	\frac{1}{1+\alpha_k} = \frac{\gamma_{k+1}}{\gamma_k+\mu \alpha_k}
	\leqslant \frac{\gamma_{k+1}}{\gamma_k},
	\]
	and it follows from \cref{eq:lambdak} that
	\begin{equation}\label{eq:lambda_alpha}
		\lambda_k \leqslant 
		\frac{\gamma_k}{\gamma_0}
		=
		\frac{L\alpha_k^2}{\gamma_0}.
	\end{equation}
	Using the lower bound of $\alpha_k$ implied by \cref{eq:lambda_alpha}, we get
	\[
	\frac{1}{\sqrt{\lambda_{k+1}}} - \frac{1}{\sqrt{\lambda_k}}
	\geqslant
	\frac{\lambda_k-\lambda_{k+1}}{2\lambda_k\sqrt{\lambda_{k+1}}}
	=\frac{\alpha_k}{2\sqrt{\lambda_k(1+\alpha_k)}}
	\geqslant \frac{1}{2}\sqrt{\frac{\gamma_0}{2L}},
	\]
	which implies 
	\[
	\frac{1}{\sqrt{\lambda_k}}\geqslant 
	\frac{k}{2}\sqrt{\frac{\gamma_0}{2L}}+1.
	\]
	Therefore, we have
	\begin{equation}\label{eq:bd-mu0}
		\lambda_k\leqslant 8L\left(2\sqrt{2L}+\sqrt{\gamma_0}k\right)^{-2}.
	\end{equation}		
	Note that this sublinear rate holds for $\mu\geqslant 0$. If $\mu>0$, 
	then by \cref{eq:low-gk} it is evident that 
\begin{equation}\label{eq:ak-mu}
	\alpha_k^2=\frac{\gamma_k}{L}\geqslant \frac{1}{L}\min\{\gamma_0,\mu\},
\end{equation}
	so we have that
	\[
	\lambda_k\leqslant 
	\left(1+\sqrt{\min\{\gamma_0,\mu\}/L}\right)^{-k}.
	\]
	This together with \cref{eq:bd-mu0} implies \cref{eq:lambdak-ex} and concludes the proof.
\end{proof}
\begin{rem}
	As we see, unlike the semi-implicit scheme \cref{eq:semi-HNAG}, explicit scheme \cref{eq:ex-HNAG} has restriction on step size $\alpha_k$.
	When $\mu>0$, namely $f$ is strongly convex, it is allowed to choose non-vanishing step size (cf. \cref{eq:ak-mu}) which promises (accelerated) linear rate. For convex $f$, i.e., $\mu=0$, \cref{eq:lambda_alpha} becomes equality which gives vanishing step size $\alpha_k = O(1/k)$ and results in accelerated sublinear rate $O(1/k^2)$.
	\end{rem}
\begin{rem}\label{rem:norm-gd}
	Note that \cref{eq:conv-ex} gives the optimal convergence rate under an oracle model of optimization complexity \cite{Nesterov:2013Introductory}. However, the explicit schemes proposed in \cite{attouch_first-order_2019,shi_acceleration_nodate} for strongly convex case ($\mu>0$) haven't achieved acceleration. In addition, we also have faster rate for the norm of gradient. Indeed, by \cref{eq:conv-ex}, we have
	\[
	\sum_{i=0}^{\infty}
	\frac{1}{\lambda_i} \nm{\nabla f(x_{i})}^2\leqslant 2L\mathcal L_0.
	\]
	This yields that 
	\[
	\min_{0\leqslant i\leqslant k}
	\nm{\nabla f(x_i)}^2\leqslant \frac{2L\mathcal L_0}{\sum_{i=0}^{k}
		1/\lambda_i },
	\]
	and asymptotically, we have $	\nm{\nabla f(x_k)}^2 = o(2L\mathcal L_0\lambda_k)$.
	On the other hand, thanks to the \cref{lem:key-est}, we have the bound
	\[
	\frac{1}{2L}	\nm{\nabla f(x_k)}^2\leqslant f(x_k)-f(x^*)\leqslant \mathcal L_k,
	\]
	which yields the uniform estimate
	\begin{equation}\label{eq:decay-Df-slow}
		\nm{\nabla f(x_k)}^2 \leqslant  2L\mathcal L_0\lambda_k.
	\end{equation}
\end{rem}
\subsection{HNAG method with one extra gradient step}
\label{sec:HNAG-gd}
Based on \cref{eq:ex-HNAG}, we propose an explicit 
scheme with one extra gradient step:
\begin{equation}\label{eq:ex-HNAG-extra}
	\left\{
	\begin{aligned}
		\frac{y_{k}-x_{k}}{\alpha_k}={}& v_{k}-y_{k}-\beta_k\nabla f(x_{k}),\\
		\frac{v_{k+1}-v_{k}}{\alpha_k}={}&
		\frac{\mu }{\gamma_k}(y_{k}-v_{k+1})
		-\frac{1}{\gamma_k}\nabla f(y_{k}),\\
		\frac{\gamma_{k+1} - \gamma_{k} }{\alpha_k}  ={}& 
		\mu -\gamma_{k+1},\\
		x_{k+1} = {}&y_k-\frac{1}{L}\nabla f(y_k),
	\end{aligned}
	\right.
\end{equation}
where $\alpha_k$ and $\beta_k$ are chosen from the relation 
\begin{equation}\label{eq:ab-extra}
	L\alpha_k^2 = \gamma_k(2+\alpha_k),
	\quad \beta_k = \frac{1}{L\alpha_k}.
\end{equation}
Below, we present this scheme in the algorithm style.
\begin{algorithm}[H]
	\caption{HNAG Method with extra gradient step }
	\label{algo:HNAG-extra}
	\begin{algorithmic}[1] 
		\REQUIRE  $\gamma_0>0$ and $x_0,v_0\in V$.
		\FOR{$k=0,1,\ldots$}
		\STATE Compute $\alpha_k, \beta_k$ by $\displaystyle 
		L\alpha_k^2 = \gamma_k(2+\alpha_k), \quad \beta_k = \frac{1}{L\alpha_k}$.
		\smallskip
		\STATE Set $\displaystyle y_{k} = \frac{1}{1+\alpha_k} \big [ x_k + \alpha_k v_k - \alpha_k\beta_k \nabla f(x_k) \big ]$.
		\smallskip			
		\STATE Update $\displaystyle v_{k+1} = 
		\frac{1}{\gamma_k+\mu\alpha_k}
		\big [\gamma_kv_k+\mu \alpha_ky_{k}- \alpha_k \nabla f(y_{k}) \big]$.
		\smallskip
		\STATE Update $\displaystyle x_{k+1} = {}y_k-\frac{1}{L}\nabla f(y_k)$.
		\smallskip
		\STATE Update $\gamma_{k+1} = (\gamma_k+\mu\alpha_k)/(1+\alpha_k)$.			
		\ENDFOR
	\end{algorithmic}
\end{algorithm}

Define 
\begin{equation*}
	\widehat{\mathcal L}_k:=
	f(y_k)-f(x^*) +
	\frac{\gamma_{k+1}}{2}
	\nm{v_{k+1}-x^*}^2.
\end{equation*}
Proceeding as the proof of \cref{thm:conv-ex1-ode-NAG}, 
we still have \cref{diff-Lk}, i.e.,
\begin{equation*}
	\begin{split}
		\widehat{\mathcal L}_{k} - \mathcal L_k
		\leqslant {}& -\alpha_k\widehat{\mathcal L}_{k} 
		-\frac{\alpha_k\beta_k}{2}\nm{\nabla f(x_k)}^2  \\
		& +\frac{1}{2}\left ( \alpha_k\beta_k - \frac{1}{L}\right )\| \nabla f(y_{k}) - \nabla f(x_k) \|^2\\
		& +\frac{1}{2}\left ( \frac{\alpha_k^2}{\gamma_k}  - \alpha_k\beta_k \right )\nm{\nabla f(y_{k})}^2.
	\end{split}
\end{equation*}
We then use our choice of parameters \cref{eq:ab-extra} to obtain
\begin{equation}\label{eq:hat-Lk-Lk}
	\widehat{\mathcal L}_k-\mathcal L_k\leqslant 
	-\alpha_k\widehat{\mathcal L}_k+
	\frac{1+\alpha_k}{2L}  \| \nabla f(y_{k})\|^2
	-\frac{1}{2L}  \| \nabla f(x_{k})\|^2,
\end{equation}
Recalling the standard gradient descent result (cf. \cite[Lemma 1.2.3]
{Nesterov:2013Introductory})
\[
f(y-\nabla f(y)/L)-f(y) \leqslant 
-\frac{1}{2L}\nm{\nabla f(y) }^2\quad\forall\,y\in V,
\]
we get the inequality
\[
\mathcal L_{k+1}-\widehat{\mathcal L}_k
=f(x_{k+1})-f(y_k)
=f(y_k-\nabla f(y_k)/L)-f(y_k)
\leqslant 	- \frac{1}{2L}\nm{\nabla f(y_k) }^2.
\]
By \cref{eq:hat-Lk-Lk}, it follows that
\begin{equation}\label{eq:diff-Lk-extra}
	\mathcal L_{k+1}-\mathcal L_k
	\leqslant -\alpha_k\mathcal L_{k+1}
	-\frac{1}{2L}  \| \nabla f(x_{k})\|^2.
\end{equation}
Hence, using the same notation as that in 
\cref{thm:conv-ex1-ode-NAG}, we have the 
following result.
\begin{thm}
	\label{thm:conv-ex1-ode-NAG-extra}
	For \cref{algo:HNAG-extra}, we have
	\begin{equation}\label{eq:diff-Ek-HNAG-extra}
		\mathcal E_{k+1}
		\leqslant 
		\frac{	 \mathcal E_k}{1+\alpha_k }\quad\forall\,k\geqslant 0.
	\end{equation}
	Hence, for all $k\geqslant0$, it holds that
	\begin{equation}\label{eq:conv-ex-extra}
		\mathcal L_k+\frac{1}{2L}\sum_{i=0}^{k-1}
		\frac{\lambda_k}{\lambda_i} \nm{\nabla f(x_{i})}^2
		\leqslant\lambda_k\mathcal L_0,
	\end{equation}
	where $\lambda_k$ is defined by \cref{eq:lambdak} and still has the optimal upper bound
	\begin{equation}\label{eq:conv-HNAG-extra}
		\lambda_k\leqslant \min\left\{
		4L\left(2\sqrt{L}+\sqrt{1.5\gamma_0}\,k\right)^{-2},\,
		\left(1+\sqrt{2\min\{\gamma_0,\mu\}/L}\right)^{-k}
		\right\}.
	\end{equation}
\end{thm}
\begin{proof}
	Note that \cref{eq:diff-Ek-HNAG-extra,eq:conv-ex-extra} have been derived from \cref{eq:diff-Rk-ex,eq:diff-Lk-extra}. The estimate \cref{eq:conv-HNAG-extra} for $\lambda_k$ follows from the procedure in \cref{thm:conv-ex1-ode-NAG} so we omit it here.
\end{proof}
\begin{rem}
	Note that the optimal convergence rate \cref{eq:conv-HNAG-extra}
	is slightly better than \cref{eq:lambdak-ex} due to an extra gradient step in \cref{algo:HNAG-extra}. However, two gradient $\nabla f(x_k)$ and $\nabla f(y_k)$ should be computed in one iteration. In \cref{algo:HNAG}, although there are still two gradient $\nabla f(x_k)$ and $\nabla f(x_{k+1})$, the later one can be re-used in the next iteration and thus essentially only one gradient is computed in one iteration. In most applications, evaluation of gradient is the dominant cost and thus \cref{algo:HNAG} is still more efficient than \cref{algo:HNAG-extra}.
\end{rem}
\subsection{Equivalence to methods from NAG flow}
\label{sec:methodfromNAG}
In this section, we shall show some explicit schemes that are supplemented with one gradient descent steps 
for NAG flow \cref{eq:NAG-sys} can be viewed as explicit discretizations for H-NAG flow \cref{eq:HNAG}.

Recall that, in \cite{luo_chen_2019_from}, we present two explicit schemes for NAG flow \cref{eq:NAG-sys}. The first one reads as follows
\begin{equation}\label{eq:ex1-ode-NAG}
	\left\{
	\begin{aligned}
		\frac{y_k-x_{k}}{\alpha_k}={}& v_{k}-y_k,\\
		\frac{v_{k+1}-v_{k}}{\alpha_k}={}&
		\frac{\mu}{\gamma_k}(y_k-v_{k+1})
		-\frac{1}{\gamma_k}\nabla f(y_k),\\
		x_{k+1} = {}&y_k-\frac{1}{L}\nabla f(y_k),\\
		\gamma_{k+1}={}&\gamma_k+\alpha_k(\mu-\gamma_{k+1}).
	\end{aligned}
	\right.
\end{equation}
Let us represent $x_k$ from the first equation 
\[
x_{k} = y_k+\alpha_k(y_k-v_k),
\]
and put this into the third equation to obtain 
\[
y_{k+1}+\alpha_{k+1}(y_{k+1}-v_{k+1}) = y_k-\frac{1}{L}\nabla f(y_k).
\]
Now reorganizing \cref{eq:ex1-ode-NAG} yield that
\[
\left\{
\begin{aligned}
\frac{v_{k+1}-v_{k}}{\alpha_k}={}&
\frac{\mu}{\gamma_k}(y_k-v_{k+1})
-\frac{1}{\gamma_k}\nabla f(y_k),\\
\frac{y_{k+1}-y_{k}}{\alpha_{k+1}}={}&
v_{k+1}-y_{k+1}
-\beta_{k+1}\nabla f(y_k),
\end{aligned}
\right.
\]
where $\beta_{k+1}=1/(L\alpha_{k+1})$. This is nothing but an explicit scheme for \cref{eq:HNAG}. In addition, writing the previous iteration for $y_{k}$ before $v_{k+1}$ and replacing $y_k$ with $x_{k+1}$ yield
\[
\left\{
\begin{aligned}
\frac{x_{k+1}-x_{k}}{\alpha_{k}}={}&
v_{k}-x_{k+1}
-\beta_{k}\nabla f(x_{k}),\\
\frac{v_{k+1}-v_{k}}{\alpha_k}={}&
\frac{\mu}{\gamma_k}(x_{k+1}-v_{k+1})
-\frac{1}{\gamma_k}\nabla f(x_{k+1}),
\end{aligned}
\right.
\]
which is identical to the scheme \cref{eq:ex-HNAG} but with slightly different choice of parameters. If $L\alpha_k^2 = \gamma_k(1+\alpha_k)$, then by \cite[Theorem 2]{luo_chen_2019_from}, we have the optimal convergence rate
\[
\mathcal L_{k}\leqslant 
\mathcal L_{0}\times
\min\left\{
4L\big(2\sqrt{L}+\sqrt{\gamma_0}\, k\big)^{-2},
\left(1+\sqrt{\min\{\gamma_0,\mu\}/L}\right)^{-k}
\right\},
\]
where $\mathcal L_k$ is defined by \cref{eq:Lk}.

The second scheme is listed below
\begin{equation}\label{eq:ex2-ode-NAG}
	\left\{
	\begin{aligned}
		\frac{y_{k}-x_{k}}{\alpha_k}={}& 
		\frac{\gamma_k}{\gamma_{k+1}}(v_{k}-y_k),\\
		\frac{v_{k+1}-v_{k}}{\alpha_k}={}&
		\frac{\mu}{\gamma_{k+1}}(y_k-v_{k})
		-\frac{1}{\gamma_{k+1}}\nabla f(y_k),\\
		x_{k+1} = {}&y_k-\frac{1}{L}\nabla f(y_k),\\
		\gamma_{k+1}={}&\gamma_k+\alpha_k(\mu-\gamma_{k}),
	\end{aligned}
	\right.
\end{equation}
which recoveries Nesterov's optimal method \cite[Chapter 2]{Nesterov:2013Introductory} constructed by estimate sequence.
Proceeding as before, we can eliminate $\{x_k\}$ and rearrange \cref{eq:ex2-ode-NAG} by that
\[
\left\{
\begin{aligned}
\frac{v_{k+1}-v_{k}}{\alpha_k}={}&
\frac{\mu}{\gamma_{k+1}}(y_k-v_{k})
-\frac{1}{\gamma_{k+1}}\nabla f(y_k),\\
\frac{y_{k+1}-y_{k}}{\alpha_{k+1}}={}&
\frac{\gamma_{k+1}}{\gamma_{k+2}}(v_{k+1}-y_{k+1})
-\beta_{k+1}\nabla f(y_k)
,\\
\gamma_{k+1}={}&\gamma_k+\alpha_k(\mu-\gamma_{k}),
\end{aligned}
\right.
\]
where $\beta_{k+1}=1/(L\alpha_{k+1})$. This is also an explicit scheme for our H-NAG flow \cref{eq:HNAG}. If $L\alpha_k^2 = \gamma_{k+1}$, then by \cite[Theorem 3]{luo_chen_2019_from}, we have the optimal convergence rate
\[
\mathcal L_{k}\leqslant 
\mathcal L_{0}\times
\min\left\{
4L\big(2\sqrt{L}+\sqrt{\gamma_0}\, k\big)^{-2},
\left(1-\sqrt{\min\{\gamma_0,\mu\}/L}\right)^{k}
\right\},
\]
which indicates the decay of the norm of gradient, i.e.,
\begin{equation}\label{eq:conv-Df-2}
	\nm{\nabla f(x_k)}^2\leqslant 2L
	\mathcal L_{0}\times
	\min\left\{
	4L\big(2\sqrt{L}+\sqrt{\gamma_0}\, k\big)^{-2},
	\left(1-\sqrt{\min\{\gamma_0,\mu\}/L}\right)^{k}
	\right\}.
\end{equation} 

We conclude that H-NAG flow offers us a better explanation and understanding for Nesterov's optimal method \cite[Chapter 2]{Nesterov:2013Introductory} than NAG flow \cref{eq:NAG-sys} does and in view of \cref{rem:norm-gd,eq:conv-Df-2}, algorithms based on H-NAG yields faster decay for the norm of the gradient. 
\section{Splitting Schemes with Accelerated Rates}
\label{sec:split}
In this section, we consider the composite case $f = h+g$ and assume that $f \in\mathcal S_{\mu}^{0}$ with $\mu\geqslant 0$, $h\in\mathcal F_L^1$ is the smooth part and the nonsmooth part $g$ is convex and lower semicontinuous. Note that this assumption on $f$ is more general than that in \cite{luo_chen_2019_from,Nesterov_2012,Siegel:2019}.
To utilize the composite structure of $f$, we shall consider splitting schemes that are explicit in $h$ and implicit in $g$ and prove the accelerated convergence rates.
\subsection{Analysis of \cref{algo:HNAG-comp}}
It is easy to show Algorithm \ref{algo:HNAG-comp} can be written as a splitting scheme 
\begin{equation}\label{eq:ex-HODE}
\left\{
\begin{aligned}
\frac{x_{k+1}-x_{k}}{\alpha_k}\in{}& v_{k}-x_{k+1}-\beta_k\nabla h(x_{k})-\beta_k\partial g(x_{k+1}),\\
\frac{v_{k+1}-v_{k}}{\alpha_k}={}&
\frac{\mu }{\gamma_k}(x_{k+1}-v_{k+1})
-\frac{1}{\gamma_k}\left(\nabla h(x_{k+1})+p_{k+1}\right),\\
\frac{\gamma_{k+1} - \gamma_{k} }{\alpha_k}  ={}& 
\mu -\gamma_{k+1},
\end{aligned}
\right.
\end{equation}
where $\alpha_k$ and $\beta_k$ are chosen from \cref{eq:ab}, i.e.,
\begin{equation}\label{eq:split-ab}
\alpha_k = \sqrt{\frac{\gamma_k}{L}},\quad\beta_k = \frac{1}{L\alpha_k},
\end{equation}
and the term $p_{k+1}$ is defined as follows
\[
p_{k+1} := \frac{1}{\beta_k}\left(v_{k}-x_{k+1}-\beta_k\nabla h(x_{k})-\frac{x_{k+1}-x_{k}}{\alpha_k}\right)
\in\partial g(x_{k+1}).
\]
If we introduce
\[
y_k:=\frac{x_k+\alpha_kv_k}{1+\alpha_k} ,\quad s_k:=\frac{\alpha_k\beta_k}{1+\alpha_k},\quad
z_{k} :=y_k- s_k \nabla h(x_k),
\]
then the update of $x_{k+1}$ in \cref{eq:ex-HODE} is equivalent to 
\begin{equation*}
\begin{split}
x_{k+1} = {}&\mathop{\argmin}\limits_{y\in V}
\left(h(x_k)+\dual{\nabla h(x_k),y-x_k}+g(y)+\frac{1}{2s_k}\nm{y-y_k}^2\right)\\
={}&\proxi_{s_k g}(y_k- s_k \nabla h(x_k)).
\end{split}
\end{equation*}

\begin{thm}
	\label{thm:conv-split}
	For \cref{algo:HNAG-comp}, we have
	\begin{equation}\label{eq:conv-split}
	\mathcal L_{k+1}
	\leqslant 
	\frac{	 \mathcal L_k}{1+\alpha_k }\quad\forall\, k\geqslant 0,
	\end{equation}
	where $\mathcal L_k$ is defined in \cref{eq:Lk},
	and it holds that
	\begin{equation}\label{eq:conv-algo2}
	\mathcal L_k
	\leqslant\mathcal L_0\times \min\left\{
	8L\left(2\sqrt{2L}+\sqrt{\gamma_0}k\right)^{-2},\,
	\left(1+\sqrt{\min\{\gamma_0,\mu\}/L}\right)^{-k}
	\right\}.
	\end{equation}
\end{thm}
\begin{proof}
	Based on the equivalent form \cref{eq:ex-HODE}, the proof is almost identical to a combination of that of  \cref{thm:conv-semi} and \cref{thm:conv-ex1-ode-NAG}.
	Let us start from the difference 
	\begin{align*}
	\mathcal L_{k+1} - \mathcal L_k
	={}&  \mathcal L(x_{k+1}, v_{k},\gamma_{k}) - \mathcal L(x_k, v_{k},\gamma_{k}) \\
	&+ \mathcal L(x_{k+1}, v_{k+1},\gamma_{k}) - \mathcal L(x_{k+1}, v_{k},\gamma_{k}) \\
	&+\mathcal L(x_{k+1}, v_{k+1},\gamma_{k+1}) - \mathcal L(x_{k+1}, v_{k+1},\gamma_{k}) \\
	:={}& {\rm I}_1 + {\rm I}_2 + {\rm I}_3,
	\end{align*}
	where the estimates for ${\rm I}_2$ and ${\rm I}_3$ keep unchanged
	\begin{align}
	{\rm I}_3 ={}& \alpha_k (\nabla_{\gamma}\mathcal L(\bs x_{k+1}), \mathcal G^{\gamma}(\bs x_{k+1})),\nonumber\\
	{\rm I}_2 \leqslant{}&  \alpha_k \dual{\nabla_v \mathcal L(\bs x_{k+1}), \mathcal G^v(\bs x_{k+1})} - \frac{\gamma_k}{2}\nm{ v_{k+1} - v_k}^2.\nonumber
	\end{align}
	Observing that 
	\begin{equation*}
	{\rm I}_1 =  \mathcal L(x_{k+1}, v_{k},\gamma_{k}) - \mathcal L(x_k, v_{k},\gamma_{k}) 
	= g(x_{k+1})- g(x_{k})+h(x_{k+1})- h(x_{k}),
	\end{equation*}
	we use \cref{lem:key-est} and the fact $p_{k+1}\in \partial g(x_{k+1})$ to estimate ${\rm I}_1$
	\begin{equation}\label{eq:I1}
	{\rm I}_1 \leqslant{}
	\dual{\nabla h(x_{k+1})+p_{k+1}, x_{k+1} - x_k} 
	- \frac{1}{2L}\nm{\nabla h(x_{k+1}) - \nabla h(x_k)}^2.
	\end{equation}
	For simplicity, set $q_{k+1} = p_{k+1}+\nabla h(x_{k+1})\in\partial f(x_{k+1})$. We use the discretization \cref{eq:ex-HODE} to replace $x_{k+1} - x_k$ and compare with the flow evaluated at $\bs x_{k+1}=(x_{k+1}, v_{k+1}, \gamma_{k+1})$:
	\begin{align*}
	\dual{q_{k+1}, x_{k+1} - x_k} 
	={}&  \alpha_k \dual{q_{k+1}, \mathcal G^x(\bs x_{k+1},\beta_k)} \\
	& +\alpha_k\beta_k \dual{q_{k+1}, \nabla h(x_{k+1}) - \nabla h(x_k)}\\
	& +\alpha_k \dual{q_{k+1}, v_k - v_{k+1}}.
	\end{align*}
	The last term is estimated in the same way as \cref{eq:df-diff-vk}, namely,
	\[
	\alpha_k \|q_{k+1}\|\| v_k - v_{k+1}\|\leqslant {}
	\frac{\alpha_k^2}{2\gamma_k} \|q_{k+1}\|^2 
	+ \frac{\gamma_k}{2}\| v_k - v_{k+1}\|^2.
	\]
	Thanks to the negative term in \cref{eq:I1}, we bound the second term by that
	\[
	\alpha_k\beta_k \dual{q_{k+1}, \nabla h(x_{k+1}) - \nabla h(x_k)}\leqslant 
	\frac{1}{2L}\nm{\nabla h(x_{k+1}) - \nabla h(x_k)}^2
	+\frac{L\alpha_k^2\beta_k^2}{2}\nm{q_{k+1}}^2.
	\]
	We now get the estimate for ${\rm I}_1$ as follows
	\begin{equation*}
	\begin{aligned}
	{\rm I}_1 \leqslant{}& \alpha_k \dual{q_{k+1}, \mathcal G^x(\bs x_{k+1},\beta_k)}  +\frac{\gamma_k}{2}\| v_k - v_{k+1}\|^2 +\left(\frac{L\alpha_k^2\beta_k^2}{2}
	+\frac{\alpha_k^2}{2\gamma_k}\right)\nm{q_{k+1}}^2.
	\end{aligned}
	\end{equation*}
	Putting all together and using \cref{lem:LG-non} implies
	\begin{equation}\label{eq:diff-Lk-split}
	\begin{split}
	\mathcal L_{k+1} - \mathcal L_k
	\leqslant {}&\alpha_k ( \partial \mathcal L(\bs x_{k+1}, q_{k+1}), \mathcal G(\bs x_{k+1},\beta_k,q_{k+1}))\\
	{}&\quad+\left(\frac{L\alpha_k^2\beta_k^2}{2}
	+\frac{\alpha_k^2}{2\gamma_k}\right)\nm{q_{k+1}}^2\\
	\leqslant {}&- \alpha_k \mathcal L_{k+1}+\left(\frac{L\alpha_k^2\beta_k^2}{2}
	+\frac{\alpha_k^2}{2\gamma_k}-\alpha_k\beta_k\right)\nm{q_{k+1}}^2\\
	={}& - \alpha_k \mathcal L_{k+1},
	\end{split}
	\end{equation}
	where in the last step we used the fact \cref{eq:split-ab}. This establishes \cref{eq:conv-split} and yields that $\mathcal L_k\leqslant \lambda_k\mathcal L_0$. Note the bound \cref{eq:lambdak-ex} for $\lambda_k$ still holds here and \cref{eq:conv-algo2} follows directly. We finally conclude the proof of this theorem.
\end{proof}
\begin{rem}\label{rem:alter-split}
To control the sub-gradient, we can choose
	\[
	\alpha_k = \sqrt{\frac{\gamma_k}{4L}},\quad \beta_k = \frac{1}{2L\alpha_k}.
	\]
	Plugging this into \cref{eq:diff-Lk-split} indicates
	\begin{equation}\label{eq:diff-Lk-split-rem}
	\mathcal L_{k+1} - \mathcal L_k
	\leqslant - \alpha_k \mathcal L_{k+1}
	-\frac{\alpha_k\beta_k}{2}\nm{q_{k+1}}^2.
	\end{equation}
	By slight modification of the proof, it follows that 
	\begin{equation}\label{eq:conv-split-rem-}
	\mathcal L_k+\frac{1}{4L}\sum_{i=0}^{k-1}
	\frac{\lambda_k}{\lambda_i} \nm{q_{i+1}}^2
	\leqslant\lambda_k\mathcal L_0.
	\end{equation}
	Following the estimate for $\lambda_k$ in \cref{thm:conv-ex1-ode-NAG}, we can derive that
	\[
	\lambda_k\leqslant \min\left\{
	32L\left(4\sqrt{2L}+\sqrt{\gamma_0}k\right)^{-2},\,
	\left(1+0.5\sqrt{\min\{\gamma_0,\mu\}/L}\right)^{-k}
	\right\}.
	\]
	Therefore, \cref{eq:conv-split-rem-} yields fast convergence for the norm of (sub-)gradient. However, the convergence bound is slightly worse than that of \cref{eq:conv-algo2}.
\end{rem}
\subsection{Methods using gradient mapping}
\label{sec:HNAG-gm}
In \cite{luo_chen_2019_from}, using the gradient mapping technique \cite{Nesterov_2012}, we presented two explicit schemes (supplemented with one gradient descent step) for NAG flow \cref{eq:NAG-sys} in composite case $f = h+g$, 
where $h\in\mathcal S_{\mu,L}^{1,1}$ with $\mu\geqslant 0$, $g$ is convex and lower-semicontinuous. Following the discussion in \S \ref{sec:methodfromNAG}, 
we show that those two schemes can also be viewed as explicit discretizations for H-NAG flow \cref{eq:HNAG}. 

Since the argument of those two methods are analogous, 
we only consider the following algorithm \cite[Algorithm 2]{luo_chen_2019_from}
\begin{equation}\label{eq:ex1-ode-NAGs}
\left\{
\begin{aligned}
\frac{y_k-x_{k}}{\alpha_k}={}& v_{k}-y_k,\\
\frac{v_{k+1}-v_{k}}{\alpha_k}={}&
\frac{\mu}{\gamma_k}(y_k-v_{k+1})
-\frac{1}{\gamma_k}\widehat{\nabla}f(y_k),\\
x_{k+1} = {}&y_k-\frac{1}{L}\widehat{\nabla}f(y_k),\\
\gamma_{k+1}={}&\gamma_k+\alpha_k(\mu-\gamma_{k+1}).
\end{aligned}
\right.
\end{equation}
Above, the gradient mapping 
$\widehat{\nabla}f(y_k)$ is defined as follows
\begin{equation}\label{eq:def-GF}
\begin{split}
\widehat{\nabla}f(y_k):={}&L
\Big(y_k -\proxi_{s g}\big[y_k-\nabla h(y_k)/L\big]\Big)\\
\in{}&\nabla h(y_k)+\partial g\Big(
\proxi_{g/L}\big[y_k-\nabla h(y_k)/L\big]\Big).
\end{split}
\end{equation}
If $L\alpha_k^2 = \gamma_k(1+\alpha_k)$, then by \cite[Theorem 4]{luo_chen_2019_from}, we have the accelerated convergence rate
\[
\mathcal L_{k}\leqslant 
\mathcal L_{0}\times
\min\left\{
4L\big(2\sqrt{L}+\sqrt{\gamma_0}\, k\big)^{-2},
\left(1+\sqrt{\min\{\gamma_0,\mu\}/L}\right)^{-k}
\right\},
\]
where $\mathcal L_k$ is defined in \cref{eq:Lk}. With a similar simplify process as that in Section \ref{sec:methodfromNAG}, we can eliminate the sequence $\{x_k\}$ and obtain the following equivalent form of \cref{eq:ex1-ode-NAGs}:
\[
\left\{
\begin{aligned}
\frac{v_{k+1}-v_{k}}{\alpha_k}={}&
\frac{\mu}{\gamma_k}(y_k-v_{k+1})
-\frac{1}{\gamma_k}\widehat{\nabla}f(y_k),\\
\frac{y_{k+1}-y_{k}}{\alpha_{k+1}}={}&
v_{k+1}-y_{k+1}
-\beta_{k+1}\widehat{\nabla}f(y_k),
\end{aligned}
\right.
\]
where $\beta_{k+1}=1/(L\alpha_{k+1})$. This is indeed an explicit scheme for H-NAG flow \cref{eq:HNAG}. Writing the previous iteration for $y_{k}$ before $v_{k+1}$ and replacing $y_k$ with $x_{k+1}$ yield
\[
\left\{
\begin{aligned}
\frac{x_{k+1}-x_{k}}{\alpha_{k}}={}&
v_{k}-x_{k+1}
-\beta_{k}\widehat{\nabla}f(x_{k}),\\
\frac{v_{k+1}-v_{k}}{\alpha_k}={}&
\frac{\mu}{\gamma_k}(x_{k+1}-v_{k+1})
-\frac{1}{\gamma_k}\widehat{\nabla}f (x_k),
\end{aligned}
\right.
\]
which is almost identical to \cref{eq:ex-HODE}. The difference is that the gradient mapping uses
$$
\widehat{\nabla}f(x_k) = \nabla h(x_k) + \partial g(\hat{x}_{k}),
$$
where $\hat{x}_{k} = \proxi_{g/L}\big[x_k-\nabla h(x_k)/L\big]$, while the scheme \cref{eq:ex-HODE} considers
$$
\nabla h(x_k) + \partial g(x_{k+1}) \quad\text{ and }\quad
\nabla h(x_{k+1}) + \partial g(x_{k+1}).
$$

\section{Conclusion and Future Work}
\label{sec:con}
In this paper, for convex optimization problem, we present a novel DIN system, which is called Hessian-driven Nesterov accelerated gradient flow. Convergence of the continuous trajectory and algorithm analysis are established via tailored Lyapunov functions satisfying the strong Lyapunov property (cf. \cref{eq:A}). It is proved that explicit schemes posses 
the optimal(accelerated) rate
\[
O\left(
\min
\left\{
1/k^2,\,\big(1+\sqrt{\mu/L}\big)^{-k}
\right\}
\right),
\]
and fast control of the norm of gradient is also obtained. 
This together with our previous work in  \cite{luo_chen_2019_from}, has already positively answered the fundamental question addressed in \cite{shi_acceleration_nodate}, that {\it we can  systematically and provably obtain accelerated methods via the numerical discretization of ordinary differential equations.}

In future work, we plan to extend our results along this line and develop a systematic framework of developing and analyzing first-order accelerated optimization methods.
	
	\section*{Acknowledgments}
	Hao Luo would like to acknowledge the support from China Scholarship Council (CSC) joint Ph.D. student scholarship (Grant 201806240132).
	
\bibliographystyle{siamplain}

\end{document}